\documentclass[a4paper,10pt,leqno]{amsart}
\usepackage[total={6in,9in},
top=1in, left=1in, right=1in, bottom=1in]{geometry}
\usepackage{amsmath}
\usepackage[italian, english]{babel}
\usepackage{amsfonts}
\usepackage{amssymb}
\usepackage{dsfont}
\usepackage{mathrsfs}
\usepackage{graphicx}
\usepackage{setspace}
\usepackage{fancyhdr}
\usepackage{amsthm}
\usepackage{empheq}
\usepackage{cases}
\usepackage[all]{xy}
\usepackage{stmaryrd}

\newcommand{\erre}{\mathds{R}}

\newcommand{\enne}{\mathds{N}}

\newcommand{\diver}{\operatorname{div}}

\newcommand{\ra}{\rightarrow}

\newcommand{\set}[1]{{\left\{#1\right\}}}               
\newcommand{\pa}[1]{{\left(#1\right)}}                  
\newcommand{\sq}[1]{{\left[#1\right]}}                  
\newcommand{\abs}[1]{{\left|#1\right|}}                 
\newcommand{\pair}[1]{\left\langle#1\right\rangle}      

\newcommand{\eps}{\varepsilon}                           

\renewcommand{\hat}[1]{\widehat{#1}}
\renewcommand{\tilde}[1]{\widetilde{#1}}

\newtheorem{theorem}{\textbf{Theorem}}[section]
\newtheorem{lemma}[theorem]{\textbf{Lemma}}
\newtheorem{proposition}[theorem]{\textbf{Proposition}}
\newtheorem{cor}[theorem]{\textbf{Corollary}}
\newtheorem{defi}[theorem]{\textbf{Definition}}
\theoremstyle{remark}
\newtheorem{rem}[theorem]{\textbf{Remark}}

\newtheorem{exe}[theorem]{\textbf{Example}}
\title[Nonexistence results in bounded domains]
{Nonexistence results for elliptic differential inequalities with
a potential in bounded domains}

\date{\today} 



\begin{document}

\maketitle

\begin{center}
\textsc{\textmd{D. D.
Monticelli\footnote{Politecnico di Milano, Italy.
Email: dario.monticelli@polimi.it.} and F.
Punzo\footnote{Universit\`{a} della Calabria, Italy. Email:
fabio.punzo@unical.it. \\ The two authors are supported by GNAMPA
projects and are members
of the Gruppo Nazionale per l'Analisi Matematica, la
Probabilit\`{a} e le loro Applicazioni (GNAMPA) of the Istituto
Nazionale di Alta Matematica (INdAM). }, }}
\end{center}
\begin{abstract}
In this paper we are concerned with a class of elliptic
differential inequalities with a potential in bounded domains both of $\erre^m$ and
of Riemannian manifolds. In particular, we investigate the effect
of the behavior of the potential at the boundary of the domain on nonexistence of nonnegative
solutions.
\end{abstract}


\section{Introduction}
In this paper we investigate nonexistence of nonnegative solutions to elliptic
differential inequalities of the form
\begin{equation}\label{EQ_gen}
 \frac{1}{a(x)}\diver\pa{a(x)\abs{\nabla u}^{p-2}\nabla u}+V(x)u^\sigma
 \leq0 \quad \textrm{in}\;\; \Omega,
 \end{equation}
where $\Omega$ is an open relatively compact connected subset of a general $m-$dimensional Riemannian manifold $M$ endowed with a metric tensor $g$, $\nabla$, $\diver$ and $\Delta$ denote the gradient, the divergence operator and the Laplace-Beltrami operator associated to the metric, respectively.
Furthermore, here and
in the rest of the paper we assume that $a : \Omega \ra \erre$
satisfies
\begin{equation}\label{EQ_propr_a}
  a>0, \quad a \in \operatorname{Lip}_{\text{loc}}(\Omega),
\end{equation}
$V>0$ a.e. on $\Omega,\, V\in L^1_{\text{loc}}(\Omega)$, and the
constants $p$ and $\sigma$ satisfy $p>1$, $\sigma>p-1$. Tipically $V$ is {\it unbounded} at $\partial \Omega$. We explicitly note
that some of the results we find are new also for the model equation
\begin{equation}\label{EQ_lapl}
 \Delta u+V(x)u^\sigma \leq 0\quad \textrm{in}\;\; \Omega,
\end{equation}
in the special case $\Omega\subset \mathbb R^m$.

\smallskip

If, differently from what will be the focus of the present paper, we consider the case when $\Omega=\mathbb R^m$ or $\Omega=M$, where $M$ is a complete noncompact Riemannian manifold, then there exists an extensive literature concerning nonexistence of nonnegatve solutions of equation \eqref{EQ_gen}.
We refer to \cite{DaMit},
\cite{MitPohoz359}, \cite{MitPohoz227}, \cite{MitPohozAbsence} and
\cite{MitPohozMilan} for a comprehensive description of results
related to these (and also more general) problems on $\erre^m$.
Note that analogous results have also been obtained for degenerate
elliptic equations and inequalities (see, e.g., \cite{DaLu},
\cite{Mont}), and for the parabolic companion problems (see, e.g.,
\cite{MitPohozMilan}, \cite{PoTe}, \cite{P1}, \cite{PuTe}).  The results in the case of a complete Riemannian manifold have a
more recent history, in particular we cite the inspiring papers \cite{GrigKond} and \cite{GrigS}, and the papers \cite{MMP1}, \cite{MMP2}, \cite{Sun1}, \cite{Sun2}.
In particular it is showed that equation
\eqref{EQ_gen} admits the unique nonnegative solution $u\equiv 0$, assuming certain
key assumptions hold, which are concerned with the parameters $p$, $\sigma$ and with the behavior of a suitably weighted volume of geodesic balls, with density given by the product of $a$ and of a negative power of the potential $V$.

\smallskip

In this work we intend to focus our attention on the case where $\Omega\subset M$ is an open relatively compact domain, considering {\it local weak solutions}, meant in the sense of Definition \ref{defsol} below.

We start with a definition describing the weighted volume growth conditions on special subsets of $\Omega$, contained in a neighborhood of $\partial \Omega$, that will be used in obtaining our nonexistence
results for nonnegative solutions of \eqref{EQ_gen}. We denote the canonical Riemannian measure on
$M$ with $d\mu_0$, while we define
\begin{equation}\label{meas_a}
d\mu = a\, d\mu_0
\end{equation}
the weighted measure on $M$ with density $a$.

Let $d(x):=\operatorname{dist}(x,\partial\Omega)$ for any $x\in \bar \Omega$\,.
For every $\delta>0$ we define
\begin{equation}\label{ed1}
\mathcal S^\delta :=\{ x\in \Omega\,:\,d(x) < \delta \},\qquad\qquad \Omega^\delta:=\Omega\setminus\overline{\mathcal S^\delta}.
\end{equation}

Recall that $p>1$, $\sigma>p-1$, $V>0$ a.e. in $\Omega$ and $V\in
  L^1_{\text{loc}}(\Omega)$ and define
  \begin{equation}\label{31}
  \alpha=\frac{p\sigma}{\sigma-p+1},  \qquad \beta=\frac{p-1}{\sigma-p+1}.
  \end{equation}
 Note that $\alpha>1$. We introduce the following three weighted volume growth conditions:
\begin{itemize}
    \item[i)] We say that \emph{condition (HP1)} holds if there exist $C>0$, $C_0>0$, $k\in [0, \beta)$, $\delta_0\in (0,1)$, $\eps_0\in \left(0, \frac{\alpha-1}{C_0} \wedge \beta\right)$ such that, for every $\delta\in (0, \delta_0)$  and for every $\eps\in (0, \eps_0),$
    \begin{equation}\label{EQ_hp1}
      \int_{\mathcal S^{\delta}\setminus \mathcal S^{\delta/2}} V^{-\beta+\eps}\,d\mu \leq C \delta^{\alpha-C_0\eps}|\log \delta|^k.
    \end{equation}
    \item[ii)] We say that \emph{condition (HP2)} holds if  there exist $C>0$, $C_0>0$, $\delta_0\in (0,1)$, $\eps_0\in \left(0, \frac{\alpha-1}{C_0}\wedge \beta \right)$ such that, for every $\delta\in (0, \delta_0)$  and for every $\eps\in (0, \eps_0)$,
    \begin{equation}\label{EQ_hp2}
      \int_{\mathcal S^{\delta}\setminus \mathcal S^{\delta/2}} V^{-\beta+\eps}\,d\mu \leq C \delta^{\alpha - C_0\eps}|\log \delta|^\beta \qquad \text{ and } \qquad \int_{\mathcal S^{\delta}\setminus \mathcal S^{\delta/2}} V^{-\beta-\eps}\,d\mu \leq C \delta^{\alpha - C_0\eps}|\log \delta|^\beta.
    \end{equation}
       \item[iii)] We say that \emph{condition (HP3)} holds if there exist $C>0$, $C_0\geq0$, $k\geq 0$, $\theta>0$, $\tau>\max\{\frac{\sigma-p+1}{\sigma}\pa{k+1},1\}$, $\delta_0\in (0, 1)$, $\eps_0\in \left(0, \frac{\alpha-1}{C_0}\wedge \beta\right)$ such that, for every  $\delta\in (0, \delta_0)$  and for every $\eps\in (0, \eps_0)$,
    \begin{equation}\label{EQ_hp3}
      \int_{\mathcal S^{\delta}\setminus \mathcal S^{\delta/2}} V^{-\beta+\eps}\,d\mu \leq C \delta^{\alpha-C_0\eps}|\log \delta|^k e^{-\eps \theta |\log
      \delta|^\tau}.
    \end{equation}
  \end{itemize}

\begin{rem}
Observe that, in general, conditions $(HP1),\, (HP2)$ and $(HP3)$
are mutually independent. In particular, we note that, in general,
$(HP3)$ does not imply $(HP1)$;
this is essentially due to the fact that constant $C$ in $(HP3)$
must be independent of $\delta$ and $\varepsilon$, see Example \ref{exindip}. The remaining cases
can be easily treated; we leave the details to the interested
reader.
\end{rem}

Let us discuss some sufficient conditions for \eqref{EQ_hp1}, \eqref{EQ_hp2}, \eqref{EQ_hp3}\,.
\begin{rem}
\begin{itemize}
    \item[i)] Suppose that there exist $C>0, C_0>0$, $k\in (0, \beta), \delta_0\in (0,1)$ such that
    \begin{equation}\label{e30}
    V(x) \leq C \pa{d(x)}^{-C_0}\qquad\text{ for all } x\in \mathcal S^{\delta_0}
    \end{equation}
    and
    \begin{equation}\label{e31}
    \int_{\mathcal S^{\delta}\setminus \mathcal S^{\delta/2}} V^{-\beta}\,d\mu \leq C \delta^{\alpha}|\log \delta|^k
    \end{equation}
    for every $\delta\in (0, \delta_0)$, then condition \eqref{EQ_hp1} holds.
    \item[ii)] If
    \begin{equation}\label{e31es}
    a(x) \leq  C \quad \textrm{and}\quad V(x) \geq C \pa{d(x)}^{-(\sigma+1)}|\log d(x)|^{-\frac k{\beta}} \quad\text{ for all } x\in\mathcal S^{\delta_0}\,,
    \end{equation}
    then \eqref{EQ_hp1} holds, with $C_0\geq \sigma +1$.

    \item[iii)] Suppose that there exist $C>0$, $C_0>0$, $\delta_0\in (0,1)$ such that
    \begin{equation}\label{e32}
    \frac{1}{C}\pa{d(x)}^{C_0} \leq V(x) \leq C \pa{d(x)}^{-C_0}\qquad\text{ for all } x\in\mathcal S^{\delta_0}
    \end{equation}
   and
  \begin{equation}\label{e33}
  \int_{\mathcal S^{\delta}\setminus \mathcal S^{\delta/2}} V^{-\beta}\,d\mu \leq C \delta^{\alpha}|\log \delta|^\beta
   \end{equation}
  for every $\delta\in (0, \delta_0)$, then condition \eqref{EQ_hp2} holds. Moreover, condition \eqref{e33} is satisfied, provided that
    \begin{equation}\label{e31esb}
    a(x) \leq  C \quad \textrm{and}\;\; V(x) \geq C \pa{d(x)}^{-(\sigma+1)}|\log d(x)|^{-1}\qquad\text{ for all } x\in \mathcal S^{\delta_0} \,.
    \end{equation}
    Hence, if \eqref{e31esb} and the second inequality of \eqref{e32} are satisfied, then condition \eqref{EQ_hp2} holds,

    \item[iv)] Suppose there exist $C>0$, $C_0\geq0$, $k\geq 0$, $\theta>0$, $\tau>\max\{\frac{\sigma-p+1}{\sigma}\pa{k+1},1\}$, $\delta_0\in (0,1)$  such that
    \begin{equation}\label{e34}
    V(x) \leq C \pa{d(x)}^{- C_0}e^{-\theta |\log d(x)|^\tau}\qquad\text{ for all } x\in \mathcal S^{\delta_0}
    \end{equation}
    and
    \begin{equation}\label{e35}
     \int_{\mathcal S^\delta\setminus \mathcal S^{\delta/2}} V^{-\beta}\,d\mu \leq C \delta^{\alpha}|\log \delta|^k
    \end{equation}
    for every $\delta\in (0, \delta_0)$, then condition \eqref{EQ_hp3} holds.

    \item[v)] If, for some $\epsilon_1\in(0,\epsilon_0\wedge\beta),$
    \[ a(x) \leq  C \quad \textrm{and}\quad  V(x) \geq C d(x)^{-(\sigma +1)} |\log(d(x))|^{-\frac k {\beta}} e^{\left(-1+\frac{\beta}{\beta-\epsilon_1}\right)\theta|\log(d(x)|^{\tau}}\quad\text{ for all } x\in \mathcal S^{\delta_0}\,,\]
    then condition \eqref{EQ_hp3} holds.
  \end{itemize}
\end{rem}

We can now state our main theorem.

\begin{theorem}\label{thm_intro3}
Let $p>1$, $\sigma>p-1$, $V\in L^1_{\text{loc}}(\Omega)$ with  $V>0$
a.e. in $\Omega$ and  $a \in \operatorname{Lip}_{\text{loc}}(\Omega)$ with
$a>0$ on $\Omega$. Assume that one of the
conditions (HP1), (HP2) or (HP3) holds. If $u \in W^{1,p}_{\rm{loc}}(\Omega) \cap
L^\sigma_{\rm{loc}}(\Omega, V d\mu)$ is a nonnegative weak solution
of \eqref{EQ_gen}, then $u\equiv0$ in $\Omega$.
\end{theorem}

We remark that, for the case $p=2$, the weighted
volume growth conditions that we assume on geodesic balls are in
many cases sharp. In particular, in this direction we construct a counterexample in geodesic balls of Riemannian models (see Section \ref{ccex}). In order to construct such counterexample, we will provide some conditions implying that the infimum of the spectrum of the Laplace operator on
the space $L^2_V(\Omega):=\{f: \Omega \to \mathbb R\,\, \textrm{measurable such that}\,\, \int_\Omega f^2\, V d\mu \, <\, \infty\}$ is $0$, a result that can be of independent interest (see Remark \ref{rem2}). Such a spectrum is clearly related to the eigenvalue equation
\begin{equation}\label{eigeneq}
\Delta \phi + \lambda V(x) \phi =0 \quad \textrm{in}\;\; \Omega\,.
 \end{equation}

\smallskip

The rest of the paper is organized as follows. In Section
\ref{sec2} we establish some preliminary technical results,
that we put to use in Section \ref{sec3}, where we give the proof
of Theorem \ref{thm_intro3}. Finally in Section
\ref{cex} we provide a family of counterexamples.
\vspace{0,3cm}

%
%
%
%
%
%
%
%
%
%
%
%
%

\section{Preliminary results}\label{sec2}

For any relatively compact domain $D\subset \Omega$ and for any $p>1$,
$W^{1, p}(D)$ is  the completion of the space of Lipschitz
functions $w: D \ra \erre$ with respect to the norm
\[
\|w\|_{W^{1,p}(D)}= \pa{\int_D \abs{\nabla
w}^p\,d\mu_0+\int_D \abs{w}^p\,d\mu_0}^{\frac 1p}.
\]

For any function $u:\Omega\ra\erre$ we say that $u \in
W^{1,p}_{\text{loc}}(\Omega)$ if for every relatively compact domain
$D\subset\subset \Omega$ one has $ u_{|_D} \in
W^{1,p}(D)$.

\begin{defi}\label{defsol}
Let $p>1$, $\sigma>p-1$, $V>0$ a.e. in $\Omega$ and $V\in
L^1_{\text{loc}}(\Omega)$. We say that $u$ is a \emph{weak solution } of
equation \eqref{EQ_gen} if $u \in W^{1,p}_{\text{loc}}(\Omega) \cap
L^\sigma_{\text{loc}}(\Omega, V d\mu_0)$ and for every $\varphi \in
W^{1,p}(\Omega) \cap L^\infty(\Omega)$, with $\varphi \geq 0$ a.e. in $\Omega$ and
compact support, one has
 \begin{equation}\label{19}
    -\int_{\Omega} a(x)\abs{\nabla u}^{p-2}\pair{\nabla u, \nabla\pa{\frac{\varphi}{a(x)}}}\,d\mu_0 + \int_{\Omega} V(x)u^\sigma \varphi\,d\mu_0 \leq 0\,.
  \end{equation}
\end{defi}

\begin{rem}\label{rem1}
  We note that, by \eqref{EQ_propr_a}, $u \in
W^{1,p}_{\text{loc}}(\Omega) \cap L^\sigma_{\text{loc}}(\Omega, V d\mu)$ is a
weak solution of \eqref{EQ_gen} if and only if it  is a weak
solution of
\[
  \diver\pa{a(x)\abs{\nabla u}^{p-2}\nabla u} + a(x) V(x)u^\sigma \leq 0 \quad \text{in }\,
  \Omega,
\]
i.e. if and only if for every $\psi \in W^{1,p}(\Omega) \cap
L^\infty(\Omega)$, with $\psi \geq 0$ a.e. in $\Omega$ and compact support,
one has
\begin{equation}\label{EQ_weakSolGen}
    -\int_{\Omega} \abs{\nabla u}^{p-2}\pair{\nabla u, \nabla \psi}\,d\mu + \int_{\Omega} V(x)u^\sigma \psi\,d\mu \leq 0,
  \end{equation}
where $d\mu$ is the measure on $M$ with density $a$, as defined in
\eqref{meas_a}.

Indeed, given any nonnegative $\psi \in W^{1,p}(\Omega) \cap L^\infty(\Omega)$
with compact support, one can choose $\varphi=a\psi$ as a test
function in \eqref{19} in order to obtain \eqref{EQ_weakSolGen}.
Similarly,  given any nonnegative $\varphi \in W^{1,p}(\Omega) \cap
L^\infty(\Omega)$ with compact support, one can insert
$\psi=\frac{\varphi}{a}$ in \eqref{EQ_weakSolGen} and find
\eqref{19}.
\end{rem}

The following two lemmas will be crucial ingredients in the proof
the Theorem \ref{thm_intro3} (for their proofs see \cite[Lemma 2.3, 2.4]{MMP1}).

\begin{lemma}\label{LE_tech1}
Let $s \geq \frac{p\sigma}{\sigma-p+1}$ be fixed. Then there exists
a constant $C>0$ such that for every $t\in\pa{0, \min\set{1, p-1}}$,
every nonnegative weak solution $u$ of equation \eqref{EQ_gen} and
every function $\varphi \in \operatorname{Lip}(\Omega)$ with compact
support and $0 \leq \varphi \leq 1$  one has
  \begin{equation}\label{EQ_2.6Lemma_tech1}
    \frac{t}{p}\int_{\Omega}{\varphi^su^{-t-1}\abs{\nabla u}^p\chi_D}\,d\mu + \frac{1}{p}\int_{\Omega}{V u^{\sigma-t}\varphi^s}\,d\mu \leq C t^{-\frac{(p-1)\sigma}{\sigma-p+1}}\int_{\Omega}{V^{-\frac{p-t-1}{\sigma-p+1}}\abs{\nabla\varphi}^{\frac{p(\sigma-t)}{\sigma-p+1}}}\, d\mu,
  \end{equation}
  where $D = \set{x\in \Omega : u(x)>0}$, $\chi_D$ is the characteristic function of
  $D$ and $d\mu$ is the measure on $M$ with density $a$, as defined in
\eqref{meas_a}.
\end{lemma}

\begin{lemma}\label{LE_tech2}
Let $s\geq\frac{2p\sigma}{\sigma-p+1}$ be fixed. Then there exists a
constant $C>0$ such that for every nonnegative weak solution $u$ of
equation \eqref{EQ_gen}, every function $\varphi \in
\operatorname{Lip}(\Omega)$ with compact support and $0 \leq \varphi \leq
1$ and every $t\in(0,\min\{1,p-1,\frac{\sigma-p+1}{2(p-1)}\})$ one
has
\begin{align}
  \label{2.10}\int_\Omega\varphi^su^\sigma V\,d\mu\leq C t^{-\frac{p-1}{p}-\frac{(p-1)^2\sigma}{p(\sigma-p+1)}}&
     \pa{\int_{\Omega\setminus K} V^{-\frac{(t+1)(p-1)}{\sigma-(t+1)(p-1)}}
     \abs{\nabla\varphi}^\frac{p\sigma}{\sigma-(t+1)(p-1)}\,d\mu}^\frac{\sigma-(t+1)(p-1)}{p\sigma}\\
  \nonumber&\pa{\int_\Omega{V^{-\frac{p-t-1}{\sigma-p+1}}\abs{\nabla\varphi}^{\frac{p(\sigma-t)}{\sigma-p+1}}}\,d\mu}^\frac{p-1}{p}
     \pa{\int_{\Omega\setminus K}\varphi^su^\sigma V\,d\mu}^\frac{(t+1)(p-1)}{p\sigma},
\end{align}
with $K=\{x\in \Omega : \varphi(x)=1\}$ and $d\mu$ is the measure on $M$
with density $a$, as defined in \eqref{meas_a}.
\end{lemma}

From Lemma \ref{LE_tech2} we immediately deduce
\begin{cor}
Under the same assumptions of Lemma \ref{LE_tech2} there exists a
constant $C>0$, independent of $u$, $\varphi$ and $t$, such that
\begin{align}\label{EQ_2.6Lemma_tech2}
   &\pa{\int_\Omega{\varphi^su^{\sigma}V\,d\mu}}^{1-\frac{\pa{t+1}\pa{p-1}}{p\sigma}} \\
   \nonumber &\,\,\leq C t^{-\frac{p-1}{p}-\frac{\pa{p-1}^2\sigma}{p\pa{\sigma-p+1}}}
   \pa{\int_\Omega{V^{-\frac{p-t-1}{\sigma-p+1}}\abs{\nabla\varphi}^{\frac{p(\sigma-t)}{\sigma-p+1}}}\, d\mu}^{\frac{p-1}{p}}
   \pa{\int_\Omega V^{-\frac{\pa{t+1}\pa{p-1}}{\sigma-\pa{t+1}\pa{p-1}}}\abs{\nabla\varphi}^{\frac{p\sigma}{\sigma-\pa{t+1}\pa{p-1}}}\,d\mu}^{\frac{\sigma-\pa{t+1}\pa{p-1}}{p\sigma}}.
\end{align}
\end{cor}


\section{Proof of Theorem \ref{thm_intro3}}\label{sec3}

We divide the proof of Theorem \ref{thm_intro3} in three cases,
depending on which of the conditions (HP1), (HP2) or (HP3) is
assumed to hold.

\begin{proof}[Proof of Theorem \ref{thm_intro3}\,.] $(a)$ Assume that condition (HP1) holds (see \eqref{EQ_hp1}). For any
fixed $\delta\in (0,\delta_0)$ let $t:= - \frac{1}{\log \delta}$. Fix any
$C_1\geq\max\left\{\frac{4(C_0-p+1)}{p\sigma}, 1\right\}$ with $C_0$ as in condition
\eqref{EQ_hp1}, define for every $x\in \Omega$
\begin{equation}\label{8}
 \varphi(x)=\begin{cases}\begin{array}{ll}
   1&\quad\text{for }d(x)>\delta,\\
   \pa{\frac{d(x)}{\delta}}^{C_1t}&\quad\text{for }d(x)\leq \delta
   \end{array}\end{cases}
\end{equation}
and for $n\in\enne$
\begin{equation}\label{9}
 \eta_n(x)=\begin{cases}\begin{array}{ll}
   1&\quad\text{for }d(x)>\frac{\delta} n,\\
   \frac{2n }{\delta}d(x) -1 &\quad\text{for }\frac{\delta}{2 n}\leq d(x)\leq \frac{\delta} n,\\
   0&\quad\text{for }d(x)<\frac{\delta}{2n}.
 \end{array}\end{cases}
\end{equation}
Let
\begin{equation}\label{10}
\varphi_n(x)=\eta_n(x)\varphi(x)\qquad\text{for all }x\in \Omega;
\end{equation}
then $\varphi_n\in \operatorname{Lip}_c(\Omega)$ with
$0\leq\varphi_n\leq1$. Moreover, we have
$$\nabla\varphi_n=\eta_n\nabla\varphi+\varphi\nabla\eta_n\qquad \text{a.e. in } \Omega\,,$$ and for
every $b\geq1$ $$|\nabla\varphi_n|^b\leq
2^{b-1}\pa{|\nabla\varphi|^b+\varphi^b|\nabla\eta_n|^b}\qquad
\text{a.e. in } \Omega.$$ Now we use $\varphi_n$ in formula
\eqref{EQ_2.6Lemma_tech1} of Lemma \ref{LE_tech1} with any fixed
$s\geq\frac{p\sigma}{\sigma-p+1}$ and deduce that, for some positive
constant $C$ and for every $n\in\enne$ and every small enough $t>0$,
we have
\begin{align}
 \label{20}&\int_\Omega{V u^{\sigma-t}\varphi_n^s}\,d\mu\\
 \nonumber &\,\,\leq Ct^{-\frac{(p-1)\sigma}{\sigma-p+1}}\int_\Omega{V^{-\frac{p-t-1}{\sigma-p+1}}\abs{\nabla\varphi_n}^{\frac{p(\sigma-t)}{\sigma-p+1}}}\,d\mu\\
 \nonumber &\,\,= C  t^{-\frac{(p-1)\sigma}{\sigma-p+1}} \int_\Omega
   V^{-\beta+\frac{t}{\sigma-p+1}}\abs{\nabla\varphi_n}^{\frac{p(\sigma-t)}{\sigma-p+1}}d\mu\\
 \nonumber&\,\,\leq Ct^{-\frac{(p-1)\sigma}{\sigma-p+1}}2^{\frac{p(\sigma-t)}{\sigma-p+1}-1}
   \sq{\int_\Omega{V^{-\beta+\frac{t}{\sigma-p+1}}\abs{\nabla\varphi}^{\frac{p(\sigma-t)}{\sigma-p+1}}}\,d\mu+
   \int_{\mathcal S^{\frac{\delta}n}\setminus
   \mathcal S^{\frac{\delta}{2n}}}{V^{-\beta+\frac{t}{\sigma-p+1}}\varphi^{\frac{p(\sigma-t)}{\sigma-p+1}}}\abs{\nabla\eta_n}^{\frac{p(\sigma-t)}{\sigma-p+1}}\,d\mu}\\
 \nonumber&\,\,\leq Ct^{-\frac{(p-1)\sigma}{\sigma-p+1}}\sq{I_1+I_2},
\end{align}
where
\begin{align*}
 I_1&:= \int_{\Omega\setminus \mathcal S^\delta}V^{-\beta+\frac{t}{\sigma-p+1}}|\nabla\varphi|^{\frac{p(\sigma-t)}{\sigma-p+1}}d\mu,\\
 I_2&:=  \int_{\mathcal S^{\frac{\delta}{n}}\setminus \mathcal S^{\frac{\delta}{2n}}}\varphi^{\frac{p(\sigma-t)}{\sigma-p+1}}|\nabla\eta_n|^{\frac{p(\sigma-t)}{\sigma-p+1}}V^{-\beta+\frac{t}{\sigma-p+1}}d\mu.
\end{align*}
By \eqref{8}, \eqref{9} and assumption (HP1) with
$\eps=\frac{t}{\sigma-p+1}$ (see equation \eqref{EQ_hp1}), for every
$n\in\enne$ and every small enough $t>0$ we have
\begin{align}
 \label{217a}I_2 & \leq \left(\sup_{\mathcal S^{\frac{\delta}{n}}\setminus \mathcal S^{\frac{\delta}{2n}}} \varphi\right)^{\frac{p(\sigma-t)}{\sigma-p+1}}\left(\frac {2n}{\delta}\right)^{\frac{p(\sigma-t)}{\sigma-p+1}}
   \int_{\mathcal S^{\frac{\delta}{n}}\setminus \mathcal S^{\frac{\delta}{2n}}} V^{-\beta + \frac{t}{\sigma-p+1}}\,d\mu \\
 \nonumber& \leq C \left(\frac{1}{n}\right)^{\frac{p(\sigma-t)}{\sigma-p+1}C_1 t}\left(\frac {2n}{\delta} \right)^{\frac{p(\sigma-t)}{\sigma-p+1}}
   \left(\frac{\delta}{n}\right)^{\alpha - \frac{C_0t}{\sigma-p+1}} \left|\log\left(\frac{\delta}{n}\right)\right|^k \\
 \nonumber& \leq C n^{-\alpha+\frac{C_0 t}{\sigma-p+1}+\frac{p(\sigma-t)}{\sigma-p+1}(-C_1t+1)} \delta^{\alpha-\frac{C_0 t}{\sigma-p+1}
   -\frac{p(\sigma-t)}{\sigma-p+1}}\left|\log\left(\frac{\delta}n\right)\right|^k\,.
\end{align}
By our choice of $C_1$, for every small enough $t>0$
\begin{equation}\label{217b}
- \alpha+\frac{C_0 t}{\sigma-p+1} +
\frac{p(\sigma-t)}{\sigma-p+1}(- C_1t+1)= \frac{t(C_0 -p\sigma C_1 +
pC_1 t-p)}{\sigma-p+1}\leq -\frac{t}{\sigma-p+1}<0.
\end{equation}
Moreover, since $t=-\frac{1}{\log \delta}$, we have
\[
 \delta^{\alpha-\frac{C_0 t}{\sigma-p+1}
   -\frac{p(\sigma-t)}{\sigma-p+1}}=\delta^{\frac{-C_0+p}{\sigma-p+1}t}=e^{\frac{- C_0+p}{\sigma-p+1}t\log
 \delta}=e^{\frac{C_0 - p}{\sigma-p+1}}
\]
In view of \eqref{217a} and \eqref{217b} for $\delta>0$ small enough, and
thus $t=-\frac 1{\log \delta}$ small enough, we obtain
\begin{equation}\label{218}
I_2 \leq C n^{-\frac{t}{\sigma-p+1}}\left|\log\left(\frac{\delta}{n}\right)\right|^k\,.
\end{equation}

In order to estimate $I_1$, we need the next

{\bf Claim:}\,\,  If
$f:(0,\infty)\rightarrow[0,\infty)$ is a nonnegative decreasing measurable
function and \eqref{EQ_hp1} holds, then for any
$\eps\in (0, \eps_0)$ and for any $\delta\in (0, \delta_0)$ we have
\begin{equation}\label{219d}
\int_{\mathcal S^\delta} f(d(x)) \pa{V(x)}^{-\beta+\eps}\,d\mu\leq
C\int_{0}^{\frac{\delta}{2}} f(r)r^{\alpha-C_0\eps-1}|\log r|^k\,dr
\end{equation}
for some positive constant $C$. In fact, we have that
\begin{equation*}
\begin{aligned}
\int_{\mathcal S^\delta} f\, V^{-\beta + \eps} d\mu & = \sum_{n=0}^\infty \int_{\mathcal S^{\frac{\delta}{2^n}} \setminus \mathcal S^{\frac{\delta}{2^{n+1}}} } f(d(x)) V^{-\beta + \eps}\, d\mu\\
&  \leq \sum_{n=0}^\infty f\left(\frac{\delta}{2^{n+1}}\right)\int_{\mathcal S^{\frac{\delta}{2^n}} \setminus \mathcal S^{\frac{\delta}{2^{n+1}}} } V^{-\beta + \eps} d\mu \\
 & \leq C \sum_{n=0}^\infty f\left(\frac{\delta}{2^{n+1}} \right)\left(\frac{\delta}{2^n} \right)^{\alpha-C_0\eps}\left|\log\left(\frac{\delta}{2^n} \right) \right|^k
\\ & \leq C \sum_{n=0}^\infty f\left(\frac{\delta}{2^{n+1}}\right)\left(\frac{\delta}{2^{n+2}}\right)^{\alpha- C_0\eps-1}\frac{\delta}{2^{n+1}}\left|\log\left(\frac{\delta}{2^{n+2}}\right) \right|^k\\
&\leq C \sum_{n=0}^\infty \int_{\frac{\delta}{2^{n+2}}}^{\frac{\delta}{2^{n+1}}} f(r) r^{\alpha-C_0\eps-1} |\log r|^k \, dr = C \int_0^{\frac{\delta}2} f(r) r^{\alpha- C_0\eps-1} |\log r|^k dr\,.
\end{aligned}
\end{equation*}
Hence the claim has been shown. Moreover, there holds
\begin{equation}\label{219a}
|\nabla \varphi(x)| \leq C_1 t \,\delta^{- C_1 t} \, (d(x))^{C_1 t -1}\quad \textrm{a.e. in}\:\: \Omega\,.
\end{equation}
Thus, using \eqref{219d}, \eqref{219a},
\begin{align*}
  I_1 &= \int_{\mathcal S^\delta} V^{-\beta+\frac{t}{\sigma-p+1}}(\delta^{-C_1 t}C_1 t (d(x))^{C_1t-1})^{\frac{p(\sigma-t)}{\sigma-p+1}}\, d\mu\\
  &\leq C  \delta^{-\frac{p(\sigma-t)}{\sigma-p+1}C_1 t}   \int_0^{\frac{\delta}2} t^{\frac{p(\sigma-t)}{\sigma-p+1}}
     r^{\frac{p(\sigma-t)}{\sigma-p+1}(C_1t-1)+\alpha-C_0\frac{t}{\sigma-p+1}-1}|\log r|^k\, dr.
\end{align*}
Now note that $$\delta^{-\frac{p(\sigma-t)}{\sigma-p+1}C_1
t}=e^{\frac{p(\sigma-t)}{\sigma-p+1}C_1}<e^\frac{p\sigma
C_1}{\sigma-p+1}$$ and that by our choice of $C_1$ we have
\begin{equation}\label{ed30}
\begin{aligned}
a&:= \frac{p(\sigma-t)}{\sigma-p+1}(C_1t-1)+\alpha-C_0\frac{t}{\sigma-p+1}\\
&\,=\frac t{\sigma-p+1}(-pC_1 t + p\sigma C_1 +p - C_0)\geq \frac{t}{2(\sigma-p+1)}>0\,.
\end{aligned}
\end{equation}
Then, by the above inequalities and performing the change of
variables $\xi:=a\log r$, we get
\begin{align}
  \label{11}I_1& \leq C t^{\frac{p(\sigma-t)}{\sigma-p+1}}\int_{0}^{\frac{\delta}{2}}
     r^{a}(- \log r)^k\,\frac{dr}{r}\\
  \nonumber&\leq C a^{-(k+1)} t^{\frac{p(\sigma-t)}{\sigma-p+1}} \int_{-\infty}^0 e^{\xi} (-\xi)^k\, d\xi\\
  \nonumber&\leq C\pa{\frac{t}{2(\sigma-p+1)}}^{-k-1}t^{\frac{p(\sigma-t)}{\sigma-p+1}}\\
  \nonumber&\leq C t^{\frac{p(\sigma-t)}{\sigma-p+1}-k-1}\,.
\end{align}
By \eqref{20}, \eqref{218} and \eqref{11}
\begin{align}
  \label{219g}\int_{\Omega\setminus\mathcal S^\delta} V u^{\sigma-t}\, d\mu\,& \leq \int_{\Omega} V u^{\sigma-t}\varphi_n^s\, d\mu\\
  \nonumber& \leq C t^{-\frac{(p-1)\sigma}{\sigma-p+1}}\left[ n^{-\frac{t}{\sigma-p+1}}\left|\log\left(\frac{\delta}{n}\right)\right|^k
    + t^{\frac{p(\sigma-t)}{\sigma-p+1}-k-1}\right].
\end{align}
Since $\delta>0$ is small and fixed, and thus $t=-\frac 1{\log \delta} < 1$
is also fixed, taking the $\liminf$ as $n\to \infty$ in
\eqref{219g} we obtain
\begin{equation}\label{219h}
\int_{\Omega\setminus \mathcal S^\delta} V u^{\sigma-t} d\mu \leq C t^{\frac{p(\sigma-t)}{\sigma-p+1}-k-1-\frac{(p-1)\sigma}{\sigma-p+1}}.
\end{equation}
Now observe that, for each small enough $t>0$,
$$\frac{p(\sigma-t)}{\sigma-p+1}-k-1-\frac{(p-1)\sigma}{\sigma-p+1} = \frac{p-1}{\sigma-p+1}-k-\frac{pt}{\sigma-p+1}=\beta-k-\frac{pt}{\sigma-p+1}\geq \delta_*>0\,.$$
Then, for any fixed sufficiently small $t>0$, we have
\[\int_{\Omega} V u^{\sigma-t}\chi_{\Omega\setminus \mathcal S^{e^{-1/t}}}\,d\mu = \int_{\Omega\setminus\mathcal S^{e^{-1/t}}} V u^{\sigma-t}\, d\mu \leq C t^{\delta_*}\,. \]
By Fatou's Lemma, taking the $\liminf$ as $t\to 0^+$ in the previous
inequality we obtain
\[ \int_\Omega V u^\sigma\, d\mu \leq 0,\]
which implies $u\equiv 0$ in $\Omega$.

\medskip

\noindent $(b)$ Assume that condition (HP2) holds (see
\eqref{EQ_hp2}). Let the functions $\varphi$, $\eta_n$ and
$\varphi_n$ be defined on $M$ as in formulas \eqref{8}, \eqref{9}
and \eqref{10}, with $\delta\in(0,\delta_0)$, $t=-\frac{1}{\log \delta}$,
$C_1\geq\max\set{\frac{4(C_0-p+1)}{p\sigma}, 1,\frac{2(p+C_0)}{\sigma-p+1}}$
and $C_0$ as in condition \eqref{EQ_hp2}. We now apply formula
\eqref{EQ_2.6Lemma_tech2}, using the family of functions
$\varphi_n\in \operatorname{Lip}_0(M)$ and any fixed
$s\geq\frac{2p\sigma}{\sigma-p+1}$. Therefore, we get
\begin{align*}
   \pa{\int_{\Omega}{\varphi_n^su^{\sigma}V\,d\mu}}^{1-\frac{\pa{t+1}\pa{p-1}}{p\sigma}} \leq& C t^{-\frac{p-1}{p}-\frac{\pa{p-1}^2\sigma}{p\pa{\sigma-p+1}}}
   \pa{\int_{\Omega}{V^{-\frac{p-t-1}{\sigma-p+1}}\abs{\nabla\varphi_n}^{\frac{p(\sigma-t)}{\sigma-p+1}}}\, d\mu}^{\frac{p-1}{p}}\\
   \nonumber &\quad\times
   \pa{\int_{\Omega} V^{-\frac{\pa{t+1}\pa{p-1}}{\sigma-\pa{t+1}\pa{p-1}}}\abs{\nabla\varphi_n}^{\frac{p\sigma}{\sigma-\pa{t+1}\pa{p-1}}}\,d\mu}^{\frac{\sigma-\pa{t+1}\pa{p-1}}{p\sigma}}.
\end{align*}
We now need need to estimate
  \begin{equation}\label{EQ_2.11in2_a}
    \int_{\Omega} V^{-\frac{p-t-1}{\sigma-p+1}}\abs{\nabla
    \varphi_n}^{\frac{p(\sigma-t)}{\sigma-p+1}}\,d\mu\qquad\text{ and
    }\qquad
    \int_{\Omega} V^{-\frac{(t+1)(p-1)}{\sigma-\pa{t+1}\pa{p-1}}}\abs{\nabla \varphi_n}^{\frac{p\sigma}{\sigma-\pa{t+1}\pa{p-1}}}\,d\mu.
  \end{equation}
Arguing as in the first part of the proof of the theorem, under the validity
of condition (HP1), with the only difference that the condition
$k<\beta$ there is replaced here by $k=\beta$, using \eqref{EQ_hp2}
we can deduce that
  \begin{equation}\label{15}
  \int_{\Omega} V^{-\frac{p-t-1}{\sigma-p+1}}\abs{\nabla \varphi_n}^{\frac{p(\sigma-t)}{\sigma-p+1}}\,d\mu\leq C\sq{n^{-\frac{t}{\sigma-p+1}}\left|\log\left(\frac{\delta}{n}\right)\right|^\beta
  +t^{\frac{p(\sigma-t)}{\sigma-p+1}-\beta-1}}.
  \end{equation}
In order to estimate the second integral in \eqref{EQ_2.11in2_a} we
start by defining $\Lambda = \frac{(p-1)\sigma
t}{(\sigma-p+1)\sq{\sigma-\pa{t+1}\pa{p-1}}}$, and we note that
\begin{equation}\label{12}
   \frac{(p-1)\sigma}{\pa{\sigma-p+1}^2}t<\Lambda<\frac{2(p-1)\sigma}{\pa{\sigma-p+1}^2}t<\eps^*
\end{equation}
for every small enough $t>0$, and that
  \[\frac{(t+1)(p-1)}{\sigma-\pa{t+1}\pa{p-1}}=\beta+\Lambda
  \qquad\text{ and }\qquad\frac{p\sigma}{\sigma-\pa{t+1}\pa{p-1}}=\alpha +\Lambda
  p\,.
  \]
By our definition of the functions $\varphi_n$, for every $n\in\enne$ and
every small enough $t>0$ we have
\begin{align}
  \label{14}\int_\Omega V^{-\beta-\Lambda}\abs{\nabla \varphi_n}^{\alpha+\Lambda p}\,d\mu& \leq
    C\sq{\int_\Omega V^{-\beta-\Lambda}{\eta_n}^{\alpha+\Lambda p}\abs{\nabla \varphi}^{\alpha+\Lambda p}\,d\mu
      +\int_\Omega V^{-\beta-\Lambda} \varphi^{\alpha+\Lambda p}\abs{\nabla \eta_n}^{\alpha+\Lambda p}\,d\mu}\\
  \nonumber  &\leq C\sq{\int_{\mathcal S^\delta} V^{-\beta-\Lambda} \abs{\nabla \varphi}^{\alpha+\Lambda p}\,d\mu
      +\int_{\mathcal S^{\frac{\delta}{n}}\setminus \mathcal S^{\frac{\delta}{2n}}} V^{-\beta-\Lambda} \varphi^{\alpha+\Lambda p}\abs{\nabla\eta_n}^{\alpha+\Lambda p}\,d\mu}\\
   \nonumber &:=C\pa{I_1+I_2}.
\end{align}
Now we use condition \eqref{EQ_hp2} with $\eps=\Lambda$, and we
obtain
   \begin{align*}
    I_2&= \int_{\mathcal S^{\frac{\delta}{n}}\setminus \mathcal S^{\frac{\delta}{2n}}} V^{-\beta-\Lambda} \varphi^{\alpha+\Lambda p}\abs{\nabla \eta_n}^{\alpha+\Lambda p}\,d\mu\\
       &\leq \pa{\sup_{\mathcal S^{\frac{\delta}{n}}\setminus \mathcal S^{\frac{\delta}{2n}}}\varphi}^{\alpha+\Lambda p}\pa{\frac{2n}{\delta}}^{\alpha+\Lambda p}\pa{ \int_{\mathcal S^{\frac{\delta}{n}}\setminus \mathcal S^{\frac{\delta}{2n}}}
           V^{-\beta-\Lambda}\,d\mu} \\
    &\leq Cn^{-\pa{\alpha+\Lambda p}C_1t}\pa{\frac{2n}{\delta}}^{\alpha+\Lambda p}\left(\frac{\delta}{n} \right)^{\alpha-C_0\Lambda}\left|\log\left(\frac{\delta}{n}\right)\right|^\beta \\
    &\leq C n^{-\pa{\alpha+\Lambda p}C_1t + p\Lambda +C_0\Lambda }\delta^{-p\Lambda - C_0\Lambda}\left|\log\left(\frac{\delta}{n}\right)\right|^\beta.
   \end{align*}
By our definition of $\Lambda$ and \eqref{12}, choosing $C_1>0$ big enough, we
easily find
   \begin{align}
     \label{13} -C_1\pa{\alpha+\Lambda p}t + \Lambda( p+C_0) &< -C_1 \alpha t -\frac{C_1 p  (p-1)\sigma}{(\sigma-p+1)^2} t^2 + \frac{2(p+C_0)(p-1)\sigma}{(\sigma-p+1)^2}t \\
     \nonumber&\leq- \hat C t<0,
   \end{align}
for some $\hat C>0,$ for any small enough $t>0$. Moreover by \eqref{12}, since
$t=-\frac{1}{\log \delta}$, we have
\[
\delta^{-p\Lambda-C_0\Lambda}\leq
\delta^{-\frac{2(p-1)\sigma (C_0+p) t}{(\sigma-p-1)^2}}=e^\frac{2(p-1)\sigma
(C_0+p)}{(\sigma-p-1)^2}.
\]
Thus, for any sufficiently small $\delta>0$,
\begin{equation}\label{EQ_EstI2}
    I_2 \leq C n^{ -\hat C t }\left| \log\left(\frac{\delta}{n}\right)\right|^\beta.
\end{equation}
In order to estimate $I_1$ we note that if
$f:[0,\infty)\rightarrow[0,\infty)$ is a nonnegative decreasing measurable
function and \eqref{EQ_hp2} holds, then for any small enough
$\eps>0$ and $\delta>0$ we have
\begin{equation}\label{219}
\int_{\mathcal S^\delta} f(d(x)) \pa{V(x)}^{-\beta-\eps}\,d\mu\leq
C\int_{0}^{\frac{\delta}2}f(r)r^{\alpha-C_0\eps-1}|\log
r|^\beta\,dr
\end{equation}
for some positive constant $C$, see \eqref{219d}. Thus, using \eqref{219a}, for every
small enough $t>0$ we have
  \begin{align*}
    I_1 &\leq \int_{\mathcal S^\delta} V^{-\beta-\Lambda}\pa{C_1 t \delta^{-C_1t} (d(x))^{C_1t-1}}^{\alpha+\Lambda p}\,d\mu \\
    &\leq C\int_{0}^{\frac{\delta}2} \delta^{-(\alpha+\Lambda p)C_1t} t^{\alpha+\Lambda p}
       r^{(\alpha+\Lambda p)(C_1t-1)+\alpha-C_0\Lambda-1}|\log r|^\beta\,dr \\
  \end{align*}
  Now, since $t=-\frac{1}{\log \delta}$, by  \eqref{12} we have
  \[
      \delta^{-(\alpha+\Lambda p)C_1t}= e^{(\alpha+\Lambda p)C_1}\leq
      e^{(\alpha+\eps^*p)C_1};
  \]
  moreover, in view of \eqref{13}, for some $\hat C_1>0$ and for every $t>0$ small enough
  \[
  \hat C t \leq b :=(\alpha+\Lambda
  p)(C_1t-1)+\alpha-C_0\Lambda\leq \hat C_1 t.
  \]
  With the change of variables $\xi=b\log r$, using the previous inequalities we find
  \begin{align}
    \label{EQ_EstI1}  I_1 &\leq C t^{\alpha+\Lambda p}\int_0^{\frac{\delta}2} r^b|\log r|^\beta\,\frac{dr}{r}= C t^{\alpha+\Lambda p} b^{-\beta-1}\pa{\int_{-\infty}^0
       e^{\xi}|\xi|^\beta\,d\xi}\\
    \nonumber&\leq Ct^{\alpha+\Lambda  p-\beta-1}.
  \end{align}
  From equations \eqref{14}, \eqref{EQ_EstI2} and \eqref{EQ_EstI1} it follows that
  \begin{align}\label{EQ_EstI1plusI2}
      \int_M V^{-\beta-\Lambda}\abs{\nabla \varphi_n}^{\alpha+\Lambda p}\,d\mu&\leq C\sq{t^{\alpha+\Lambda  p-\beta-1}+n^{ -\hat C t}\left|\log\left(\frac{\delta}{n}\right)\right|^\beta}.
  \end{align}
  From \eqref{EQ_2.6Lemma_tech2}, using \eqref{15} and \eqref{EQ_EstI1plusI2}  we have
  \begin{align*}
    \pa{\int_{\Omega\setminus\mathcal S^\delta}u^\sigma V\,d\mu}^{1-\frac{(t+1)(p-1)}{p\sigma}} &\leq \pa{\int_{\Omega}\varphi_n^su^\sigma V\,d\mu}^{1-\frac{(t+1)(p-1)}{p\sigma}}\\
    &\leq C t^{-\frac{p-1}{p}-\frac{\pa{p-1}^2\sigma}{p\pa{\sigma-p+1}}}\pa{\int_{\Omega}V^{-\frac{p-t-1}{\sigma-p+1}}
       \abs{\nabla\varphi_n}^{\frac{p\pa{\sigma-t}}{\sigma-p+1}}\,d\mu}^{\frac{p-1}{p}} \\
    &\,\,\,\,\,\, \times\pa{\int_\Omega V^{-\beta-\Lambda}\abs{\nabla\varphi_n}^{\alpha+\Lambda p}\,d\mu}^\frac{1}{\alpha+\Lambda p} \\
    &\leq C t^{-\frac{p-1}{p}-\frac{\pa{p-1}^2\sigma}{p\pa{\sigma-p+1}}}\sq{n^{-\frac{t}{\sigma-p+1}}\left| \log\left(\frac{\delta}{n}\right)\right|^\beta
       +t^{\frac{p(\sigma-t)}{\sigma-p+1}-\beta-1}}^{\frac{p-1}{p}} \\
    &\,\,\,\,\,\,\times\sq{t^{\alpha+\Lambda  p-\beta-1}+n^{-\hat C t}\left| \log\left(\frac{\delta}{n}\right)\right|^\beta}^{\frac{1}{\alpha+\Lambda p}}\,.
  \end{align*}
  By taking the $\liminf$ as $n\ra+\infty$ we get
  \begin{align*}
  \pa{\int_{\Omega\setminus\mathcal S^\delta} u^\sigma V\,d\mu}^{1-\frac{(t+1)(p-1)}{p\sigma}} &\leq C
    t^{-\frac{p-1}{p}-\frac{\pa{p-1}^2\sigma}{p\pa{\sigma-p+1}}+\frac{(p-1)(\sigma-t)}{\sigma-p+1}-\frac{(\beta+1)(p-1)}{p}+1-\frac{\beta+1}{\alpha+\Lambda p}}
  \end{align*}
  for every sufficiently small $t>0$, with $t=-\frac{1}{\log \delta}$. But
  \[
  -\frac{p-1}{p}-\frac{\pa{p-1}^2\sigma}{p\pa{\sigma-p+1}}+\frac{(p-1)(\sigma-t)}{\sigma-p+1}-\frac{(\beta+1)(p-1)}{p}+1-\frac{\beta+1}{\alpha+\Lambda p}
  =-\frac{(p-1)^2}{p\pa{\sigma-p+1}}t,
  \]
 hence for every small enough $t>0$ we have
  \begin{equation*}
    \pa{\int_{\Omega\setminus \mathcal S^{e^{-1/t}}}u^\sigma V\,d\mu}^{1-\frac{(t+1)(p-1)}{p\sigma}} \leq C
    t^{-\frac{(p-1)^2}{p\pa{\sigma-p+1}}t}\leq C,
  \end{equation*}
  that is
  \begin{equation*}
    \int_{\Omega\setminus \mathcal S^{e^{-1/t}}} u^\sigma V\,d\mu \leq C
  \end{equation*}
  uniformly in $t$, for $t>0$ sufficiently small. By taking the limit for $t\ra0^+$ we deduce
  \begin{equation}
    \int_{\Omega}u^\sigma V\,d\mu < +\infty,
  \end{equation}
  and thus $u\in L^\sigma(\Omega,Vd\mu)$. Now we exploit inequality \eqref{2.10} with the cutoff function
  $\varphi_n$, and using again \eqref{15} and \eqref{EQ_EstI1plusI2} we obtain
  \begin{align*}
    \int_\Omega\varphi_n^s u^\sigma V\,d\mu &\leq C t^{-\frac{p-1}{p}-\frac{\pa{p-1}^2\sigma}{p\pa{\sigma-p+1}}}
    \sq{n^{-\frac{t}{\sigma-p+1}}\left|\log\left(\frac{\delta}{n}\right)\right|^\beta+t^{\frac{p(\sigma-t)}{\sigma-p+1}-\beta-1}}^{\frac{p-1}{p}}
    \pa{\int_{\mathcal S^\delta}\varphi_n^s u^\sigma V\,d\mu}^{\frac{(t+1)(p-1)}{p\sigma}} \\
    &\,\,\,\,\,\,\times\sq{t^{\alpha+\Lambda  p-\beta-1}+n^{ -\hat C t}\left|\log\left(\frac{\delta}{n}\right)\right|^\beta}^{\frac{1}{\alpha+\Lambda p}}.
  \end{align*}
  Since $\varphi_n\equiv 1$ in $\Omega\setminus \mathcal S^\delta$ and $0\leq\varphi_n\leq 1$ in $\Omega$, for all $n\in\mathds{N}$
  \[
   \int_{\Omega\setminus\mathcal S^\delta} u^\sigma V\,d\mu\leq \int_\Omega\varphi_n^s u^\sigma V\,d\mu, \qquad \int_{\mathcal S^\delta}\varphi_n^s u^\sigma V\,d\mu \leq \int_{\mathcal S^\delta} u^\sigma V\,d\mu.
  \]
  Using previous inequalities and taking the $\liminf$ as $n\ra+\infty$ we
  get
  \begin{align*}
    \int_{\Omega\setminus\mathcal S^\delta} u^\sigma V\,d\mu &\leq C t^{-\frac{p-1}{p}-\frac{\pa{p-1}^2\sigma}{p\pa{\sigma-p+1}}+\frac{(p-1)(\sigma-t)}{\sigma-p+1}-\frac{(\beta+1)(p-1)}{p}
    +1-\frac{\beta+1}{\alpha+\Lambda p}}\pa{\int_{\mathcal S^\delta} u^\sigma V\,d\mu}^{\frac{(t+1)(p-1)}{p\sigma}} \\
    &= C t^{-\frac{(p-1)^2}{p\pa{\sigma-p+1}}t}\pa{\int_{\mathcal S^\delta} u^\sigma V\,d\mu}^{\frac{(t+1)(p-1)}{p\sigma}} \leq C \pa{\int_{\mathcal S^\delta} u^\sigma V\,d\mu}^{\frac{(t+1)(p-1)}{p\sigma}}
  \end{align*}
  uniformly for $t>0$ sufficiently small, with $t=-\frac{1}{\log \delta}$. Since $u\in L^\sigma(\Omega,Vd\mu)$,
  \[
  \int_{\mathcal S^\delta} u^\sigma V\,d\mu \ra 0 \quad \text{ as }\, \delta\ra 0.
  \]
  Moreover $\frac{(t+1)(p-1)}{p\sigma}\ra\frac{p-1}{p\sigma}>0$ as $\delta\ra 0$. It follows that
  \[
   \int_{\Omega} u^\sigma V\,d\mu = \lim_{\delta\ra 0}\int_{\Omega\setminus\mathcal S^\delta}u^\sigma V\,d\mu=0,
  \]
  which implies $u\equiv0$ in $\Omega$.

\medskip

\noindent $(c)$ Assume that condition (HP3) holds (see
\eqref{EQ_hp3}). Consider the functions $\varphi$, $\eta_n$ and
$\varphi_n$ defined in \eqref{8}, \eqref{9} and \eqref{10}, with
$\delta>0$ small enough, $t=-\frac{1}{\log \delta}$,
$C_1>0$ as in $b)$ and $C_0$ as in condition
\eqref{EQ_hp3}. Arguing as in $a)$, by formula
\eqref{EQ_2.6Lemma_tech1} with any fixed
$s\geq\frac{p\sigma}{\sigma-p+1}$, we see that
\begin{align}
 \label{2}\int_\Omega{V u^{\sigma-t}\varphi_n^s}\,d\mu& \leq
   Ct^{-\frac{(p-1)\sigma}{\sigma-p+1}}\int_\Omega{V^{-\frac{p-t-1}{\sigma-p+1}}\abs{\nabla\varphi_n}^{\frac{p(\sigma-t)}{\sigma-p+1}}}\,d\mu\\
 \nonumber&\leq Ct^{-\frac{(p-1)\sigma}{\sigma-p+1}}
   \sq{\int_\Omega{V^{-\frac{p-t-1}{\sigma-p+1}}\abs{\nabla\varphi}^{\frac{p(\sigma-t)}{\sigma-p+1}}}\,d\mu+
   \int_{\mathcal S^{\frac{\delta}{n}}\setminus
   \mathcal S^{\frac{\delta}{2n}}}{V^{-\frac{p-t-1}{\sigma-p+1}}\varphi^{\frac{p(\sigma-t)}{\sigma-p+1}}}\abs{\nabla\eta_n}^{\frac{p(\sigma-t)}{\sigma-p+1}}\,d\mu}\\
 \nonumber&:=Ct^{-\frac{(p-1)\sigma}{\sigma-p+1}}\sq{I_1+I_2},
\end{align}
for some positive constant $C$ and for every $n\in\enne$ and every
small enough $t>0$. Now, recalling the definitions of $\varphi$ and
$\eta_n$, by condition \eqref{EQ_hp3} with
$\eps=\frac{t}{\sigma-p+1}$, for every small enough $t>0$ we have
\begin{align*}
  I_2&\leq \pa{\sup_{ \mathcal S^{\frac{\delta}{n}}\setminus
   \mathcal S^{\frac{\delta}{2n}}  }\varphi}^{\frac{p(\sigma-t)}{\sigma-p+1}}\pa{\frac{2n}{\delta}}^{\frac{p(\sigma-t)}{\sigma-p+1}}
     \int_{ \mathcal S^{\frac{\delta}{n}}\setminus
   \mathcal S^{\frac{\delta}{2n}}  }V^{-\beta+\frac{t}{\sigma-p+1}}\,d\mu\\
  &\leq Cn^{-\frac{p(\sigma-t)}{\sigma-p+1}C_1t}\pa{\frac{2 n}{\delta}}^{\frac{p(\sigma-t)}{\sigma-p+1}}
     \pa{\frac{\delta}n}^{\alpha-\frac{C_0t}{\sigma-p+1}}\left| \log\left(\frac{\delta}{n}\right)\right|^k e^{-\frac{\theta t}{\sigma-p+1} \left|\log\left(\frac{\delta}{n} \right)\right|^\tau}\\
  &=C n^{-\alpha+\frac{C_0 t}{\sigma-p+1}+\frac{p(\sigma-t)}{\sigma-p+1}(-C_1t+1)} \delta^{\alpha-\frac{C_0 t}{\sigma-p+1}
   -\frac{p(\sigma-t)}{\sigma-p+1}}\left|\log\left(\frac{\delta}n\right)\right|^k    e^{-\frac{\theta t}{\sigma-p+1}\left|\log\left(\frac{\delta}{n}\right)
     \right|^\tau}.
\end{align*}
Note that since $t=-\frac{1}{\log \delta}$, we have
\[
 \delta^{\alpha-\frac{C_0 t}{\sigma-p+1}
   -\frac{p(\sigma-t)}{\sigma-p+1}}=\delta^{\frac{-C_0+p}{\sigma-p+1}t}=e^{\frac{- C_0+p}{\sigma-p+1}t\log
 \delta}=e^{\frac{C_0 - p}{\sigma-p+1}}\,.
\]
Thus, by our choice of $C_1$, if $t>0$ is sufficiently small, we
have
\begin{equation}\label{1}
I_2\leq C n^{-\frac{t}{\sigma-p+1}}\left|\log\left(\frac{\delta}{n}\right)\right|^k.
\end{equation}

In order to estimate $I_1$ we note that if
$f:[0,\infty)\rightarrow[0,\infty)$ is a nonnegative measurable decreasing
function and \eqref{EQ_hp3} holds, then for any small enough
$\eps>0$ and any sufficiently large $R>0$ we have
\begin{equation}\label{2.19}
\int_{\mathcal S^\delta} f(d(x)) \pa{V(x)}^{-\beta+\eps}\,d\mu\leq
C\int_0^{\frac{\delta}{2}}  f(r)r^{\alpha-C_0\eps-1}|\log
r|^ke^{-\eps\theta|\log r|^\tau}\,dr
\end{equation}
Inequality \eqref{2.19} can be obtained in a similar way as \eqref{219d}.

Now, using \eqref{219a} and \eqref{2.19} with
$\eps=\frac{t}{\sigma-p+1}$, we obtain that for every small enough
$t>0$ with $t=-\frac{1}{\log \delta}$
\begin{align*}
  I_1&\leq \int_{\mathcal S^\delta}
     V^{-\beta+\frac{t}{\sigma-p+1}}\pa{C_1t \delta^{- C_1t}(d(x))^{C_1t-1}}^\frac{p(\sigma-t)}{\sigma-p+1}\,d\mu\\
  &\leq  C \delta^{-\frac{p(\sigma-t)C_1t}{\sigma-p+1}}\int_0^{\frac{\delta}{2}} t^\frac{p(\sigma-t)}{\sigma-p+1}r^{\frac{p(\sigma-t)(C_1t-1)}{\sigma-p+1}+\alpha-\frac{C_0t}{\sigma-p+1}-1}|\log r|^k
     e^{-\frac{t\theta}{\sigma-p+1}|\log r|^\tau}\,dr.
\end{align*}
Note that
\[
\delta^{-\frac{p(\sigma-t)C_1t}{\sigma-p+1}}=e^{\frac{p(\sigma-t)C_1}{\sigma-p+1}}\leq e^{\frac{p\sigma
C_1}{\sigma-p+1}}.
\]
Thus, with the change of variable $r=e^{-\xi}$, we deduce
\begin{align*}
  I_1&\leq Ct^\frac{p(\sigma-t)}{\sigma-p+1}\int_0^{\frac{\delta}{2}} r^{\frac{p(\sigma-t)(C_1t-1)}{\sigma-p+1}+\alpha-\frac{C_0t}{\sigma-p+1}}|\log r|^k
     e^{-\frac{t\theta}{\sigma-p+1}|\log r|^\tau}\,r^{-1}dr.
\\
     &\leq Ct^\frac{p(\sigma-t)}{\sigma-p+1}\int_0^{+\infty} e^{-a \xi}\xi^k
        e^{-\frac{t\theta}{\sigma-p+1}\xi^\tau}\,d\xi,
\end{align*}
with $a$ defined in \eqref{ed30}. Now recall that by our choice of $C_1$, for $t>0$ small enough, we
have $a>0$. Hence, setting
$\rho=\pa{\frac{t\theta}{\sigma-p+1}}^\frac{1}{\tau}\xi$, we have
\begin{align}
  \label{3}I_1\leq Ct^\frac{p(\sigma-t)}{\sigma-p+1}\int_0^{+\infty}\xi^ke^{-\frac{t\theta}{\sigma-p+1}\xi^\tau}\,d\xi
  =Ct^\frac{p(\sigma-t)}{\sigma-p+1}\pa{\frac{t\theta}{\sigma-p+1}}^{-\frac{k+1}{\tau}}\int_0^{+\infty}\rho^ke^{-\rho^\tau}\,d\rho
  \leq Ct^{\frac{p(\sigma-t)}{\sigma-p+1}-\frac{k+1}{\tau}}.
\end{align}
From \eqref{2}, \eqref{1} and \eqref{3} we conclude that for every
$n\in\enne$ and every small enough $t=-\frac{1}{\log \delta}>0$ we have
\[
 \int_{\Omega\setminus\mathcal S^\delta}{V u^{\sigma-t}}\,d\mu\leq\int_\Omega{V u^{\sigma-t}\varphi_n^s}\,d\mu\leq
    Ct^{-\frac{(p-1)\sigma}{\sigma-p+1}}\sq{t^{\frac{p(\sigma-t)}{\sigma-p+1}-\frac{k+1}{\tau}}+n^{-\frac{t}{\sigma-p+1}}\left|\log\left(
    \frac{\delta}{n}\right)\right|^k}
\]
for some fixed positive constant $C$. Passing to the limit as
$n\rightarrow+\infty$ in the previous inequality yields
\begin{equation}\label{4}
  \int_{\Omega\setminus\mathcal S^\delta}{V u^{\sigma-t}}\,d\mu\leq
  Ct^{-\frac{(p-1)\sigma}{\sigma-p+1}+\frac{p(\sigma-t)}{\sigma-p+1}-\frac{k+1}{\tau}}.
\end{equation}
Now note that by our assumptions on $\tau,k$ we have
\[
-\frac{(p-1)\sigma}{\sigma-p+1}+\frac{p(\sigma-t)}{\sigma-p+1}-\frac{k+1}{\tau}=
\frac{\sigma}{\sigma-p+1}-\frac{k+1}{\tau}-\frac{pt}{\sigma-p+1}\geq\frac{1}{2}\pa{\frac{\sigma}{\sigma-p+1}-\frac{k+1}{\tau}}:=\delta_*>0
\]
for every small enough $t=-\frac{1}{\log \delta}>0$. Thus \eqref{4} yields
\begin{equation}\label{5}
  \int_{\Omega\setminus \mathcal S^{e^{-1/t}}}{V u^{\sigma-t}}\,d\mu\leq  Ct^{\delta_*}
\end{equation}
for every small enough $t>0$. Passing to the $\liminf$ as $t$ tends
to $0^+$ in \eqref{5}, we conclude by an application of Fatou's
Lemma that
\[
\int_{\Omega}{V u^{\sigma}}\,d\mu=0,
\]
so that $u\equiv0$ on $\Omega$. \end{proof}

\section{Counterexamples}\label{cex}
To begin, we show that in general hypothesis $(HP3)$ does not imply hypothesis $(HP1)$.

\begin{exe}\label{exindip}
Let $\sigma>1, p>1$, and let $a\in C^1(\Omega), a>0$ with
\[a(x):= \begin{cases}
1  &  \,\,  \textrm{if} \    \, d(x) >2\delta^*, \\
d(x)^{\alpha-1}|\log d(x)|^{\beta_0}e^{-\beta\theta |\log d(x)|^\tau}   & \ \textrm{if} \  d(x)\leq\delta^*   \,, \\
\end{cases}
\]
with $\delta^*>0$, $\beta_0>\beta$, $\theta>0$,
$\tau>\max\left\{\frac{\sigma-p+1}{\sigma}(\beta_0+1), 1\right\}$.
Define
\[V(x):= e^{-\theta|\log d(x)|^\tau} \quad \textrm{for all}\;\; x\in \Omega\,.\]
Then, in view of \eqref{meas_a}, for all $\delta>0$ sufficiently small and every $\varepsilon>0$
\begin{equation*}
\begin{aligned}
\int_{\mathcal S^\delta\setminus \mathcal S^{\delta/2}} V^{-\beta+\varepsilon}(x)\, d\mu &\leq e^{-\varepsilon \theta |\log
\delta|^\tau}\int_{\mathcal S^\delta\setminus \mathcal S^{\delta/2}}
e^{\beta\theta |\log d(x)|^\tau}\,d\mu\\
&\leq C e^{-\varepsilon \theta |\log \delta|^\tau}
\int_{\delta/2}^\delta r^{\alpha-1}|\log r|^{\beta_0} \,dr \,\leq\, C
e^{-\varepsilon \theta |\log \delta|^\tau}\delta^{\alpha}|\log
\delta|^{\beta_0}
\end{aligned}
\end{equation*}
for some positive constant $C$ independent of $\delta$ and
$\varepsilon$. Thus $(HP3)$ holds.

On the other hand, for every $\delta>0$ small enough and every
$\varepsilon>0$
\begin{equation}\label{ehp1}
\begin{aligned}
\int_{\mathcal S^\delta\setminus \mathcal S^{\delta/2}} V^{-\beta+\varepsilon}(x)\, d\mu& \geq
 e^{-\varepsilon \theta 2^\tau |\log
\delta|^\tau}\int_{\mathcal S^\delta\setminus \mathcal S^{\delta/2}}
e^{\beta\theta |\log d(x)|^\tau}\,d\mu\\
&\geq C e^{-\varepsilon \theta 2^\tau|\log \delta|^\tau} \int_{\delta/2}^\delta
r^{\alpha-1}|\log r|^{\beta_0} \,dr \,\geq\, C e^{-\varepsilon \theta 2^\tau
|\log \delta|^\tau}\delta ^{\alpha}|\log \delta|^{\beta_0}\,,
\end{aligned}
\end{equation}
for some positive constant $C$ independent of $\delta$ and
$\varepsilon$. Passing to the limit in \eqref{ehp1} as
$\varepsilon\to 0$ we obtain that for every $\delta>0$ small enough
\begin{equation}\label{ehp1a}
\int_{\mathcal S^\delta\setminus \mathcal S^{\delta/2}} V^{-\beta}(x)\, d\mu \geq C
\delta^{\alpha}|\log \delta|^{\beta_0}\,.
\end{equation}
If, by contradiction, $(HP1)$ holds, passing to the limit as
$\varepsilon\to 0$ in \eqref{EQ_hp1} we obtain that, for every $\delta>0$
small enough,
\begin{equation}\label{ehp1b}
\int_{\mathcal S^\delta\setminus \mathcal S^{\delta/2}} V^{-\beta}(x)\, d\mu \leq \bar C
\delta^{\alpha}|\log \delta|^{\beta}\,.
\end{equation}
Since $\beta_0>\beta$, \eqref{ehp1a} and \eqref{ehp1b} are in
contrast. So $(HP1)$ cannot hold.
\end{exe}

\medskip

Now we show that if condition $(HP1)$ or $(HP2)$ or $(HP3)$ is not satisfied, then problem \eqref{EQ_gen} can admit a positive nontrivial solution. Before constructing our counterexample, which will only deal with the case $p=2$, we need some auxiliary results on spectral theory for a \textit{weighted eigenvalue problem} for the Laplace--Beltrami operator. Moreover, we recall the notion of Riemannian model manifolds (see Section \ref{Rm}). We should mention that a similar counterexample has been constructed in \cite{GrigS} and in \cite{MMP1}, when $\Omega=M$, with $M$ a complete noncompact Riemannian manifold. However, many differences occur in the present situation, due to the fact that $\Omega$ is bounded. Furthermore, the study of the first eigenvalue for the weighted eigenvalue problem was not necessary in \cite{GrigS}, \cite{MMP1}.

\subsection{Preliminary results for a weighted eigenvalue problem}

Let $a\equiv1$, $V\in C(\Omega)$, $V>0$ in $\Omega$; for any domain $D\subseteq \Omega$ set $L^2_V(D):=\{f: D \to \mathbb R\,\, \textrm{measurable such that}\,\, \int_D f^2\, V d\mu \, <\, \infty\}$.
For every $\delta>0$, consider the {\it weighted} eigenvalue problem
\begin{equation}\label{e21}
\begin{cases}
\Delta \phi + \lambda V \phi \,=\, 0  &  \,\,  \textrm{in} \    \, \Omega^\delta, \\
 \phi\,=\,0  & \ \textrm{on} \  \partial \Omega^\delta  \,, \\
\end{cases}
\end{equation}
where $\Omega^\delta:=\Omega\setminus\overline{\mathcal S^\delta}\,.$ It is known that, since $V\in C(\overline{\Omega^\delta})$ and  $V>0$, there exist the first eigenvalue $\lambda_\delta$ and the first eigenfunction $\phi_\delta\in L^2_V(\Omega^\delta)\cap W^{1,2}_0(\Omega^\delta)$ of \eqref{e21}, with
\begin{equation}\label{e10ceeeee}
\lambda_\delta=\inf_{\phi\in C^\infty_c(\Omega^\delta), \phi\not\equiv 0} \frac{\int_{\Omega^\delta}  |\nabla \phi|^2\, d\mu}{\int_{\Omega^\delta} V \phi^2\, d\mu}\,,
\end{equation}
 see e.g. \cite{GilTru}. Moreover,
$\lambda_\delta> 0$, $\lambda_{\delta_1}\geq \lambda_{\delta_2}$ if $\delta_1>\delta_2$, and
$\lambda_\delta\to \bar\lambda(\Omega)$ as $\delta\to 0^+,$ for some $\bar \lambda(\Omega)\in [0, \infty).$ Using \eqref{e10ceeeee}, an easy computation shows that
\begin{equation}\label{e10ce}
\bar \lambda(\Omega)=\inf_{\phi\in C^\infty_c(\Omega), \phi\not\equiv 0} \frac{\int_\Omega  |\nabla \phi|^2\, d\mu}{\int_\Omega V \phi^2\, d\mu}\,.
\end{equation}

From condition \eqref{e10ce}, the following lemma immediately follows.

\begin{lemma}\label{baspec}
Let $V\in C(\Omega), V>0$ in $\Omega$. Suppose that there exists $C>0$ such that for any $\alpha>0$ there exists $\phi_\alpha\in W^{1, 2}_0(\Omega)\cap L^2_{V}(\Omega)$, with $\phi_\alpha\not\equiv0$, for which
\begin{equation}\label{e1ce}
\int_\Omega |\nabla \phi_\alpha|^2 d\mu \leq C \alpha \int_\Omega V \phi_\alpha^2 \, d\mu
\end{equation}
and
\begin{equation}\label{e1ceee}
\int_{\mathcal{S}^{2\delta}\setminus\mathcal{S}^{\delta}} |\phi_\alpha|^2 d\mu =o(\delta^2)\qquad\text{ as }\delta\to0^+\,.
\end{equation}
Then $\bar \lambda(\Omega)=0\,.$
\end{lemma}
\begin{proof}
  For any small $\delta>0$, consider a Lipschitz cut--off function $\psi$ such that $\psi\equiv0$ on $\mathcal{S}^\delta$, $\psi\equiv1$ on $\Omega\setminus\mathcal{S}^{2\delta}$ and $|\nabla\psi|\leq C\delta^{-1}$ for some $C>1$ independent of $\delta$. Then $\psi\phi_\alpha\in W^{1,2}_0(\Omega)$ with $\operatorname{supp}(\psi\phi_\alpha)\subset\Omega\setminus\mathcal{S}^\delta$ and
  \[
  \int_\Omega|\nabla(\psi\phi_\alpha)|^2\,d\mu\leq2\left(\int_\Omega\psi^2|\nabla\phi_\alpha|^2\,d\mu+\int_\Omega\phi_\alpha^2|\nabla\psi|^2\,d\mu\right)\,.
  \]
  Now note that, up to choosing $\delta>0$ small enough, since $\phi_\alpha\in L^2_{V}(\Omega)$ with $\int_\Omega\phi_\alpha^2V\,d\mu>0$,
  \[
  \int_\Omega\psi^2|\nabla\phi_\alpha|^2\,d\mu\leq\int_\Omega|\nabla\phi_\alpha|^2\,d\mu\leq C\alpha\int_\Omega\phi_\alpha^2V\,d\mu
  \leq2C\alpha\int_{\Omega^{2\delta}}\phi_\alpha^2V\,d\mu\leq 2C\alpha\int_{\Omega}(\psi\phi_\alpha)^2V\,d\mu\,.
  \]
  Moreover,
  \[
  \int_\Omega\phi_\alpha^2|\nabla\psi|^2\,d\mu\leq C\delta^{-2}\int_{\mathcal{S}^{2\delta}\setminus\mathcal{S}^{\delta}}\phi_\alpha^2\,d\mu=o(1)
  \]
  as $\delta\to0^+$. In particular, since $\phi_\alpha\in L^2_{V}(\Omega)$ with $\int_\Omega\phi_\alpha^2V\,d\mu>0$, we can choose $\delta>0$ small enough
  so that
  \[
  \int_\Omega\phi_\alpha^2|\nabla\psi|^2\,d\mu\leq \alpha\int_{\Omega\setminus\mathcal{S}^{2\delta}}\phi_\alpha^2V\,d\mu\leq \alpha\int_{\Omega}(\psi\phi_\alpha)^2V\,d\mu\,.
  \]
  We conclude that
  \[
  \int_\Omega|\nabla(\psi\phi_\alpha)|^2\,d\mu\leq C'\alpha\int_{\Omega}(\psi\phi_\alpha)^2V\,d\mu
  \]
  Since $\psi\phi_\alpha$ has compact support in $\Omega$, by standard mollification there exists $\varphi_\alpha\in C^\infty_c(\Omega)$ such that
  \begin{equation*}
      \frac{\int_{\Omega}  |\nabla\varphi_\alpha|^2\, d\mu}{\int_{\Omega} \varphi_\alpha^2V\, d\mu}\leq2\frac{\int_{\Omega}  |\nabla(\psi\phi_\alpha)|^2\, d\mu}{\int_{\Omega} (\psi\phi_\alpha)^2V\, d\mu}\leq2C'\alpha\,.
    \end{equation*}
  Hence $\bar\lambda(\Omega)\leq2C'\alpha$ for every $\alpha>0$, and the conclusion follows.
\end{proof}

Using the previous lemma, we show the next result.
\begin{proposition}\label{propce} Let $V\in C(\Omega), V>0$ in $\Omega$. Assume that
\begin{equation}\label{e8ce}
C_0 [d(x)]^{-\beta_0}  \leq V(x) \leq C_1 [d(x)]^{-\beta_1} \quad \textrm{for all}\;\; x\in \Omega\,,
\end{equation}
for some $C_1> C_0>0$, $\beta_1\geq \beta_0>2$. Then $\bar \lambda(\Omega)=0\,.$
\end{proposition}
\begin{proof}
For each $\alpha>0$ let
\[\phi_\alpha(x):=e^{-\sqrt \alpha [d(x)]^{-\gamma}},\quad x\in \Omega\,, \]
where $\gamma:=\frac{\beta_0-2}{2}.$  Note that $\phi_\alpha\in C(\Omega)$, $\phi_\alpha(x)\to 0$ as $d(x)\to 0^+$ and $\phi_\alpha$ satisfies \eqref{e1ceee}.
Furthermore, using the fact that $x\mapsto d(x)$ is Lipschitz in $\Omega$, we have
\[ \nabla \phi_\alpha(x) = \gamma \sqrt\alpha [d(x)]^{-\gamma -1} \nabla d(x) \phi_\alpha(x)\,,\quad x\in \Omega\,;  \]
therefore, since $|\nabla d(x)|\leq 1$ for a.e. $x\in \Omega$,
\[ |\nabla \phi_\alpha(x)|^2 \leq \gamma^2 \alpha [d(x)]^{-2\gamma -2} \phi_\alpha^2(x)\qquad\text{ a.e. } x\in \Omega\,. \]
Hence, in view of \eqref{e8ce},
\[\int_\Omega |\nabla \phi_\alpha(x)|^2\, d\mu \leq \gamma^2 \alpha \int_\Omega [d(x)]^{-2\gamma -2} \phi_\alpha^2(x) \, d\mu \leq
\frac{\gamma^2\alpha}{C_0}\int_\Omega V(x) \phi_\alpha^2(x) \, d\mu <\infty\,. \]
Consequently, $\phi_\alpha\in W^{1, 2}_0(\Omega)\cap L^2_V(\Omega),$ and condition \eqref{e1ce} is satisfied with $C=\frac{\gamma^2}{C_0}$. Hence, by Lemma \ref{baspec} the conclusion follows.
\end{proof}

\begin{rem}\label{rem2}
  Note that $\bar{\lambda}(\Omega)$ is the infimum of the spectrum of the Laplace operator on $L^2_V(\Omega)$, which is nonnegative and may not be achieved. Clearly, such a spectrum is deeply related to the eigenvalue equation \eqref{eigeneq}. Hence Lemma \ref{baspec} and Proposition \ref{propce} provide sufficient conditions which imply that such infimum is $0$.
\end{rem}

\subsection{Riemannian models}\label{Rm}
Let us fix a point $o\in M$ and denote by $\textrm{Cut}(o)$ the
{\it cut locus} of $o$. For any $x\in M\setminus
\big[\textrm{Cut}(o)\cup \{o\} \big]$, one can define the {\it
polar coordinates} with respect to $o$, see e.g. \cite{Grig}.
Namely, to any point $x\in M\setminus \big[\textrm{Cut}(o)\cup
\{o\} \big]$ there correspond a polar radius $r(x) := dist(x, o)$
and a polar angle $\theta\in \mathbb S^{m-1}$ such that the
shortest geodesics from $o$ to $x$ starts at $o$ with the
direction $\theta$ in the tangent space $T_oM$. Since we can
identify $T_o M$ with $\mathbb R^m$, $\theta$ can be regarded as a
point of $\mathbb S^{m-1}.$

The Riemannian metric in $M\setminus\big[\textrm{Cut}(o)\cup \{o\}
\big]$ in the polar coordinates reads as
\[ds^2 = dr^2+A_{ij}(r, \theta)d\theta^i d\theta^j, \]
where $(\theta^1, \ldots, \theta^{m-1})$ are coordinates in
$\mathbb S^{m-1}$ and $(A_{ij})$ is a positive definite matrix. It
is not difficult to see that the Laplace-Beltrami operator in
polar coordinates has the form
\begin{equation}\label{e70}
\Delta = \frac{\partial^2}{\partial r^2} + \mathcal F(r,
\theta)\frac{\partial}{\partial r}+\Delta_{S_{r}},
\end{equation}
where $\mathcal F(r, \theta):=\frac{\partial}{\partial
r}\big(\log\sqrt{A(r,\theta)}\big)$, $A(r,\theta):=\det
(A_{ij}(r,\theta))$, $\Delta_{S_r}$ is the Laplace-Beltrami
operator on the submanifold $S_{r}:=\partial B_r(o)\setminus
\textrm{Cut}(o)$ with  $B_r(o)\equiv B_r:=\{x\in M\,:\, \rho(x) <r \}$\,.

$M$ is a {\it manifold with a pole}, if it has a point $o\in M$
with $\textrm{Cut}(o)=\emptyset$. The point $o$ is called {\it
pole} and the polar coordinates $(r,\theta)$ are defined in
$M\setminus\{o\}$.

A manifold with a pole is a {\it spherically symmetric manifold} or
a {\it model}, if the Riemannian metric is given by
\begin{equation}\label{e70b}
ds^2 = dr^2+\psi^2(r)d\theta^2,
\end{equation}
where $d\theta^2$ is the standard metric in $\mathbb S^{m-1}$, and
\begin{equation}\label{26}
\psi\in \mathcal A:=\Big\{f\in C^\infty((0,\infty))\cap
C^1([0,\infty)): f'(0)=1,\, f(0)=0,\, f>0\text{ in }
(0,\infty)\Big\}.
\end{equation}
In this case, we write $M\equiv M_\psi$; furthermore, we have
$\sqrt{A(r,\theta)}=\psi^{m-1}(r)$, so the boundary area of the
geodesic sphere $\partial S_R$ is computed by
\[S(R)=\omega_m\psi^{m-1}(R),\]
$\omega_m$ being the area of the unit sphere in $\mathbb R^m$.
Also, the volume of the ball $B_R(o)$ is given by
\[\mu(B_R(o))=\int_0^R S(\xi)d\xi\,. \]
Moreover we have
\[\Delta = \frac{\partial^2}{\partial r^2}+ (m-1)\frac{\psi'}{\psi}\frac{\partial}{\partial r}+ \frac1{\psi^2}\Delta_{\mathbb S^{m-1}},\]
or equivalently
\[\Delta = \frac{\partial^2}{\partial r^2}+ \frac{S'}{S}\frac{\partial}{\partial r}+ \frac1{\psi^2}\Delta_{\mathbb S^{m-1}},\]
where $\Delta_{\mathbb S^{m-1}}$ is the Laplace-Beltrami operator in
$\mathbb S^{m-1}$.

Observe that for $\psi(r)=r$, $M=\mathbb R^m$, while for
$\psi(r)=\sinh r$, $M$ is the $m-$dimensional hyperbolic space
$\mathbb H^m$.

\subsection{Construction of a positive nontrivial solution}\label{ccex}
Let $M\equiv M_\psi$ be a model, $\Omega= B_1(o)\subset M_\psi$, $a\equiv 1$, $p=2$. Clearly, the choice $M=\mathbb R^m$ is possibile.
Let $0<\epsilon<\frac 1{\sigma -1}, C>0$ and define
\[V(x)\equiv V(r):= C(1-r)^{-(\sigma +1)}|\log(1-r)|^{-1-\epsilon(\sigma-1)},\]
with $r\equiv \rho(x).$ It is direct to see that, for some $C_1>0, C_2>0$,
\begin{equation}\label{e2ce}
C_1 \delta^{\frac{2 \sigma}{\sigma -1}}|\log \delta|^{\frac 1{\sigma -1}+\epsilon} \leq \int_{B_{1-\delta}\setminus B_{1-2\delta}} V^{-\frac 1{\sigma -1}} \, d\mu \leq C_2 \delta^{\frac{2\sigma}{\sigma -1}}|\log \delta|^{\frac 1{\sigma-1}+\epsilon}
\end{equation}
for any $\delta \in \left(0, \frac 1 2\right)$. Hence condition $(HP1)$ or $(HP2)$ or $(HP3)$ cannot be satisfied.  On the other hand, $V$ satisfies condition \eqref{e8ce}, therefore
\begin{equation}\label{e13ce}
\bar \lambda(\Omega)=0\,.
\end{equation}
We claim that there exists a positive solution of \eqref{EQ_gen}.
In order to prove the claim, we argue in three steps.

\smallskip

\noindent \textit{Step 1.}  Define
\begin{equation}\label{e6ce}
\zeta(x)\equiv \zeta(r):= (1-r) |\log(1-r)|^{\lambda} \quad \textrm{for all}\;\; x\in B_1\,.
\end{equation}
We have that
\[
\begin{aligned}
\zeta'(r)&=-|\log(1-r)|^{\lambda} +\lambda |\log(1-r)|^{\lambda-1}\,,\\
\zeta''(r)&=\frac{\lambda}{1-r}|\log(1-r)|^{\lambda -2}[(\lambda-1)-|\log(1-r)|]\,.
\end{aligned}
\]
So,
\begin{equation}\label{e4ce}
\begin{aligned}
\Delta \zeta + V(r) \zeta^\sigma = & \zeta''(r) +(m-1) \frac{\psi'}{\psi} \zeta'(r) + V(r) \zeta^\sigma(r) \\
=& \frac{|\log(1-r)|^{\lambda-2}}{1-r}\left\{\lambda(\lambda-1) -\lambda   |\log(1-r)| + \lambda(m-1)\frac{\psi'(r)}{\psi(r)} (1-r) |\log(1-r)|\right.  \\
&\,\,\,\,\left.-(m-1)\frac{\psi'(r)}{\psi(r)}(1-r)  |\log(1-r)|^2 + C  |\log(1-r)|^{(\sigma -1)(\lambda-\epsilon) +1}\right\}\,.
\end{aligned}
\end{equation}
In view of \eqref{e4ce}, if we take $0<\lambda<\epsilon$ and $\delta>0$ small enough, we get
\begin{equation}\label{e5ce}
\Delta \zeta + V \zeta^\sigma \leq 0 \quad \textrm{in}\;\; B_1\setminus B_{1-\delta}\,.
\end{equation}
Moreover, observe that
\begin{equation}\label{e60ce}
\zeta'<0 \quad \textrm{in}\;\; [r_0, 1)\,,
\end{equation}
for $r_0:=1-\delta$, if $\delta>0$ is small enough.

\smallskip

\noindent \textit{Step 2.}  For any $\rho\in (0,1)$ let $\lambda_\rho$ and $w_\rho$ be the first eigenvalue and, respectively, the first eigenfunction of problem
\begin{equation}\label{e7ce}
\begin{cases}
\Delta w_\rho + \lambda_\rho V w_\rho \,=\, 0  &  \,\,  \textrm{in} \    \, B_{\rho}, \\
 w_\rho\,=\,0  & \ \textrm{on} \  \partial B_\rho  \,,  \\
\end{cases}
\end{equation}
that is problem \eqref{e21} with $\Omega=B_1$, $\delta=1-\rho.$ It is known that $\lambda_\rho>0$ and that the corresponding eigenfunction $w_\rho$ is radial, i.e. $w_\rho=w_\rho(r),$ and does not change sign in $B_\rho$. We can suppose that $w_\rho(0)=1$; so, $w_\rho>0$ in $B_\rho$. Hence we have that
\begin{equation}\label{e11ce}
w_\rho'' + \frac{S'(r)}{S(r)} w_\rho' + \lambda_\rho V w_\rho \,=\, 0, \quad 0<r<\rho\,,
\end{equation}
with $w_\rho(\rho)=0$, $w_\rho(0)=1$, $w'_\rho(0)=0$, $w_\rho>0$ in $(0,\rho).$
From \eqref{e11ce} if follows that
\[ (S(r) w'_\rho) ' +\lambda_\rho S(r) V w_\rho =0\,.\]
Consequently, $(S w'_\rho)'\leq 0;$ so, the function $r\mapsto S(r) w'_\rho(r)$ is decreasing in $[0,\rho]$. Since it vanishes at $r=0$, we have that $S(r) w'_\rho(r)\leq 0$; therefore, $w_\rho'(r)\leq 0$ for all $r\in (0,\rho).$ Hence the function $r\mapsto w_\rho(r)$ is decreasing in $(0, \rho)$. Thus,
\begin{equation}\label{e14ce}
0< w_\rho\leq 1 \quad \textrm{in} \;\; (0, \rho)\,.
\end{equation}
Hence, $w_\rho$ is a positive solution of
\begin{equation}\label{e12ce}
\Delta w_\rho + \lambda_\rho V w_\rho^\sigma \leq 0\quad \textrm{in}\;\; B_\rho\,.
\end{equation}

We claim that there exists a sequence $\{\rho_k\}\subset (0,1)$ such that $\rho_k\to1$ and $w_{\rho_k}\to 1$ as $k\to \infty$ in $C_{\textrm{loc}}^1((0,1))$. In fact, set $\rho_n:=1-\frac 1 n$. Thanks to \eqref{e7ce} and \eqref{e14ce} with $\rho=\rho_n$, by standard elliptic regularity theory, there exists a subsequence $\{\rho_{n_k}\}\equiv\{\rho_k\}\subset \{\rho_n\}$ such that
$\{w_{\rho_k}\}$ converges in $C^\infty_{\textrm{loc}}(B_1)$ to a function $w$. Moreover, using \eqref{e13ce}, we can infer that $w$ solves
\[\Delta w = 0 \quad \textrm{in}\;\; B_1\,;\]
therefore,
\begin{equation}\label{e15ce}
\begin{cases}
w'' + \frac{S'(r)}{S(r)} w' = 0 &  \,\,  \textrm{in} \,  (0, 1), \\
w(0)=1\,. \\
\end{cases}
\end{equation}
Observe that all the solutions of the O.D.E.
\[w'' + \frac{S'(r)}{S(r)} w' = 0 \]
are given by
\[ w(r)= C_1\int_r^1 \frac{d\xi}{S(\xi)} + C_2 \quad (r \in (0, 1])
 \]
for $C_1, C_2\in \mathbb R$. However, $w(r)$ diverges as $r\to 0^+$, if $C_1\neq0$. Thus, the only bounded solution of \eqref{e15ce} is $w\equiv 1$, which corresponds to the choice $C_1=0$, $C_2=1$. Therefore, we can infer that
\begin{equation}\label{e16ce}
w_{\rho_k}\to 1\quad \textrm{in}\;\; C_{\textrm{loc}}^\infty(B_1)\;\; \textrm{as}\;\; k\to \infty\,.
\end{equation}

\smallskip

\noindent {\it Step 3\,.} Fix $\delta>0$ so that \eqref{e5ce} and \eqref{e60ce} hold, choose $\rho\in (0,1)$ such that $\rho>r_0$ and
\begin{equation}\label{e17ce}
\frac{w_\rho'(r_0)}{w_\rho(r_0)} > \frac{\zeta'(r_0)}{\zeta(r_0)}\,.
\end{equation}
This is possible, since, thanks to \eqref{e16ce},
\[ \frac{w_{\rho_k}'(r_0)}{w_{\rho_k}(r_0)} \to 0 \quad \textrm{as}\;\; k\to \infty, \]
whereas, by  \eqref{e60ce},
\[ \frac{\zeta'(r_0)}{\zeta(r_0)}<0\,.\]
Set
\[ \theta:=\inf_{[r_0, \rho)}\frac{\zeta}{w_\rho}\,.\]
Since $\lim_{r\to \rho^-} \frac{\zeta(r)}{w_\rho(r)} \to \infty\,,$ we deduce that
$\theta=\frac{\zeta(\xi)}{w_\rho(\xi)}$ for some $\xi\in [r_0, \rho).$ Actually, $\xi >r_0.$ In fact, thanks to \eqref{e17ce},
$$ \left(\frac{\zeta}{w_\rho}\right)'(r_0)=\frac{\zeta'(r_0)w_\rho(r_0) - \zeta(r_0) w'_\rho(r_0)}{w_\rho^2(r_0)} <0\,.
$$
So, $\frac{\zeta}{w_\rho}$ is strictly decreasing in a neighborhood of $r_0$ and cannot have a minimum point at $r_0.$ Hence, $\xi\in (r_0, \rho)$ and $\left(\frac{\zeta}{w_\rho}\right)'(\xi)=0.$
This implies that
\begin{equation}\label{e18ce}
\zeta(\xi)=\theta w_\rho(\xi)\,,\quad \zeta'(\xi)=\theta w_\rho'(\xi)\,.
\end{equation}
Define
\[ \tilde u(x)\equiv \tilde u(r):=
\begin{cases}
\theta w_\rho(r) &  \,\,  \textrm{for all} \  r\in (0, \xi), \\
\zeta(r) & \ \textrm{for all} \  r\in[\xi, 1). \\
\end{cases}
\]
In view of \eqref{e18ce} we have that $\tilde u\in C^1(B_1);$ hence, in particular, $\tilde u\in W^{1, 2}_{\textrm{loc}}(\Omega)$. By \eqref{e12ce},
\begin{equation}\label{e19ce}
\Delta \tilde u +\frac{\lambda_\rho}{\theta^{\sigma-1}}V \tilde u^\sigma \leq 0\quad \textrm{in}\;\; B_\xi\,.
\end{equation}
From \eqref{e5ce} and \eqref{e19ce} we obtain
\begin{equation}\label{e20ce}
\Delta \tilde u + \gamma V \tilde u^\sigma \leq 0 \quad \textrm{in}\;\; B_1\,,
\end{equation}
where $\gamma:=\min\left\{\frac{\lambda_\rho}{\theta^{\sigma-1}}, 1\right\}$. Let $u:= \gamma^{\frac 1{\sigma-1}}\tilde u$. Thanks to \eqref{e20ce} we get
\[ \Delta u + V u^\sigma \leq 0\quad \textrm{in}\;\; B_1\,.\]
So, we have exhibited a positive solution.

\end{document}